\documentclass[12pt]{article}
\usepackage[centertags]{amsmath}
\usepackage{amsfonts}
\usepackage{amssymb}
\usepackage{latexsym}
\usepackage{amsthm}
\usepackage{newlfont}
\usepackage{graphicx}
\usepackage{booktabs}
\date{}

\newlength{\defbaselineskip}
\newcommand{\setlinespacing}[1]%
           {\setlength{\baselineskip}{#1 \defbaselineskip}}

\newcommand{\disp}{\operatorname{disp}}
\newcommand{\N}{{\mathbb{N}}}

\newcommand{\actaqed}{\hfill $\actabox$}
{\medskip\noindent \textit{Proof of #1. }}%
{\actaqed \medskip}

\def\cB{\mathcal B}

\def \Tr{\mathcal T}
\def \K{\mathcal K}

\def \cF{\mathcal F}

\def\R{{\mathbb R}}
\def\Z{\mathbb Z}

\def \<{\langle}
\def\>{\rangle}

\def\ba{\mathbf a}

\def\bx{\mathbf x}
\def\by{\mathbf y}
\def\bz{\mathbf z}
\def\bk{\mathbf k}
\def\bu{\mathbf u}

\def\bm{\mathbf m}

\def\bs{\mathbf s}

\newtheorem{Theorem}{Theorem}[section]
\newtheorem{Lemma}{Lemma}[section]
\newtheorem{Definition}{Definition}[section]
\newtheorem{Proposition}{Proposition}[section]

\newtheorem{Corollary}{Corollary}[section]

\numberwithin{equation}{section}

\newcommand{\be}{\begin{equation}}
\newcommand{\ee}{\end{equation}}

\begin{document}

\title{On the fixed volume discrepancy of the Korobov point sets}
\author{A.S.  Rubtsova\thanks{Lomonosov Moscow State University, Russia.},\, K.S.  Ryutin\thanks{Lomonosov Moscow State University, Russia.}, \,and V.N. Temlyakov\thanks{University of South Carolina, USA; Steklov Institute of Mathematics and Lomonosov Moscow State University, Russia.} }
\maketitle
\begin{abstract}
This paper is devoted to the study of a discrepancy-type characteristic 
-- the fixed volume discrepancy -- 
of the Korobov point sets in the unit cube. 
It was observed recently that this new characteristic allows us to obtain 
optimal rate of dispersion from numerical integration results. 
This observation motivates us to thoroughly study this new version of 
discrepancy, which seems to be interesting by itself.  
This paper extends recent results by V. Temlyakov and M. Ullrich on the fixed volume discrepancy of the Fibonacci point sets. 
\end{abstract}

\section{Introduction} 
\label{I} 

This paper is a follow up to the recent paper \cite{VT172}. It is devoted to the study of a discrepancy-type characteristic -- the fixed volume discrepancy -- 
of a point set in the unit cube $\Omega_d:=[0,1)^d$. 
We refer the reader to the following books and survey papers on discrepancy theory and 
numerical integration \cite{BC}, \cite{Mat}, \cite{NW10}, \cite{VTbookMA}, \cite{Bi}, \cite{DTU}, \cite{T11}, and \cite{VT170}. 
 Recently, an important new observation was made in \cite{VT161}. 
It claims that a new version of discrepancy -- the $r$-smooth fixed volume discrepancy -- allows 
us to obtain optimal rate of \emph{dispersion} from numerical integration results 
(see \cite{AHR,BH19,DJ,RT,Rud,Sos,Ull,U19,UV18} for some recent results on dispersion). 
This observation motivates us to thoroughly study this new version of discrepancy, 
which seems to be interesting by itself. 

The \emph{$r$-smooth fixed volume discrepancy} takes into 
account two characteristics of a smooth hat function $h^r_B$ -- its smoothness $r$ and the volume of  its support $v:=\mathrm{vol}(B)$ (see the definition of $h^r_B$ below).  We now proceed to a formal description of the problem setting and to formulation of the results. 

Denote by $\chi_{[a,b)}(x)$  a univariate characteristic function (on $\R$) of the interval $[a,b)$ 
and, for $r=1,2,3,\dots$, we inductively define
$$
h^1_u(x):= \chi_{[-u/2,u/2)}(x)
$$
and 
$$
h^r_u(x) := h^{r-1}_u(x)\ast h^1_u(x),
$$
where
$$
f(x)\ast g(x) := \int_\R f(x-y)g(y)dy.
$$
Note that $h^2_u$ is the \emph{hat function}, i.e., 
$h_{u}^2(x) =\max\{u-|x|,0\}$.

Let $\Delta_tf(x):=f(x)-f(x+t)$ be the first difference. 
We say that a univariate function $f$ has smoothness $1$ in $L_1$ if 
$\|\Delta_t f\|_1\le C|t|$ for some absolute constant $C<\infty$. 
In case $\|\Delta^r_t f\|_1 \le C|t|^r$, where $\Delta^r_t:= (\Delta_t)^r$ is the $r$th difference operator, $r\in \N$, we say that $f$ has smoothness $r$ in $L_1$.
Then, $h^r_u(x)$ has smoothness $r$ in $L_1$ and has support $(-ru/2,ru/2)$. 

For a box $B$ of the form
\be\label{eq:B}
B= \prod_{j=1}^d [z_j-ru_j/2,z_j+ru_j/2)
\ee
 define
\be\label{eq:hB}
h^r_B(\bx):= h^r_\bu(\bx-\bz):=\prod_{j=1}^d h^r_{u_j}(x_j-z_j).
\ee
We begin with the \emph{non-periodic} $r$-smooth fixed volume discrepancy introduced and studied in \cite{VT161}.
\begin{Definition}\label{ED.1} Let $r\in\N$, $v\in (0,1]$ and $\xi:= \{\xi^\mu\}_{\mu=1}^m\subset [0,1)^d$ be a point set. 
We define the $r$-smooth fixed volume discrepancy with equal weights as
\be\label{E.1}
D^r(\xi,v):=  \sup_{B\subset \Omega_d:vol(B)=v}\left|\int_{\Omega_d} h_B^r(\bx)d\bx-\frac{1}{m}\sum_{\mu=1}^m h_B^r(\xi^\mu)\right|.
\ee
%The optimized version of the $r$-smooth fixed volume discrepancy is defined as follows
%\be\label{E.2}
%D^{r,o}(\xi,v):=  \inf_{\la_1,\dots,\la_m}\sup_{B\subset \Omega_d:vol(B)=v}\left|\int_{\Omega_d} h_B^r(\bx)d\bx- \sum_{\mu=1}^m \la_\mu h_B^r(\xi^\mu)\right|.
%\ee
\end{Definition}
\noindent

It is well known that the Fibonacci cubature formulas are optimal in the sense 
of order for numerical integration of different kind of smoothness classes of 
functions of two variables, see e.g.~\cite{DTU,TBook,VTbookMA}. 
We present a result from \cite{VT161}, which shows that the Fibonacci point set 
has good fixed volume discrepancy.

Let $\{b_n\}_{n=0}^{\infty}$, $b_0=b_1 =1$, $b_n = b_{n-1}+b_{n-2}$,
$n\ge 2$, be the Fibonacci numbers. Denote the $n$th {\it Fibonacci point set} by
$$
\cF_n:= \Big\{\big(\mu/b_n,\{\mu b_{n-1} /b_n \}\big)\colon\, \mu=1,\dots,b_n\Big\}.
$$
In this definition $\{a\}$ is the fractional part of the number $a$.   The cardinality of the set $\cF_n$ is equal to $b_n$. In \cite{VT161} we proved 
the following upper bound.

\begin{Theorem}\label{ET.1} Let $r\ge2$. 
There exist constants $c,C>0$ such that for any $v\ge c/b_n$ 
we have 
\be\label{E.3} 
D^{r}(\cF_n,v)
\,\le\, C\,\frac{\log(b_n v)}{b_n^r}.
\ee
\end{Theorem}

\bigskip
 
 The main object of  interest in the paper \cite{VT172} was the \emph{periodic} $r$-smooth $L_p$-discrepancy 
of the Fibonacci point sets. 
For this, we define the periodization $\tilde f$ 
(with period $1$ in each variable) 
of a function $f\in L_1(\R^d)$ with a compact support by
$$
\tilde f(\bx) := \sum_{\bm\in \Z^d} f(\bm+\bx)
$$
and, for each $B\subset [0,1)^d$, 
we let $\tilde h^r_B$ be the periodization of $h^r_B$ from~\eqref{eq:hB}.

We now define the periodic $r$-smooth $L_p$-discrepancy.
\begin{Definition}\label{ID.1}For $r\in\N$, $1\le p\le \infty$ and $v\in(0,1]$ define the 
periodic $r$-smooth fixed volume $L_p$-discrepancy of a point set $\xi$ by  
\be\label{1.4}
  \tilde D^{r}_p(\xi,v):=
	\sup_{B\subset \Omega_d:vol(B)=v}
	\left\|\int_{\Omega_d} \tilde h^r_B(\bx-\bz)d\bx- \frac1m\sum_{\mu=1}^m \tilde h^r_B(\xi^\mu-\bz)\right\|_p
\ee
where the $L_p$-norm is taken with respect to $\bz$ over the unit cube $\Omega_d=[0,1)^d$.\\
%Analogously to \eqref{E.2} we may define the optimized version $\tilde D^{r,o}_p(\xi,v)$.
\end{Definition}

In the case of $p=\infty$ this concept was introduced and studied in \cite{VT163}. 

The following upper bound for $1\le p<\infty$ was proved in \cite{VT172}.
\begin{Theorem}\label{IT.1} 
Let $r\in\N$ and $1\le p<\infty$. 
There exist constants $c,C>0$ such that for any $v\ge c/b_n$ 
we have 
$$
\tilde D^{r}_p(\cF_n,v) \,\le\, C\,\frac{\sqrt{\log(b_n v)}}{b_n^r}.
$$
\end{Theorem}

\medskip

In  the case $p=\infty$  a weaker upper bound was proved in \cite{VT172}.

\begin{Theorem}\label{IT.2} Let $r\ge 2$. There exist constants $c,C>0$ such that for any $v\ge c/b_n$ 
we have 
$$
\tilde D^{r}_\infty(\cF_n,v) \,\le\, C\,\frac{\log(b_n v)}{b_n^r}.
$$
\end{Theorem}

 Some comments, which show that Theorems \ref{IT.1} and \ref{IT.2} cannot be improved in a certain sense were given in \cite{VT172}.  
 
 Our main interest in this paper is to study the Korobov cubature formulas instead of the 
 Fibonacci cubature formulas from the point of view of the fixed volume discrepancy. We prove a conditional result under assumption that the Korobov cubature formulas are exact on a certain subspace of trigonometric polynomials with frequencies  from a hyperbolic cross. 
 There are results that guarantee existence of such cubature formulas (see Section \ref{F} for a discussion). 
 
 Let $m\in\N$, $\mathbf a := (a_1,\dots,a_d)$, $a_1,\dots,a_d\in\Z$.
We consider the cubature formulas
$$
P_m (f,\mathbf a):= m^{-1}\sum_{\mu=1}^{m}f\left (\left \{\frac{\mu a_1}
{m}\right\},\dots,\left \{\frac{\mu a_d}{m}\right\}\right),
$$
which are called the {\it Korobov cubature formulas}.  In the case $d=2$, $m=b_n$, $\mathbf a = (1,b_{n-1})$ we have
$$
P_m (f,\mathbf a) =\Phi_n (f):= \frac{1}{b_n}\sum_{\by\in \cF_n} f(\by).
$$

Denote 
$$
\mathbf y^{\mu}:=\left (\left \{\frac{\mu a_1}
{m}\right\},\dots,\left \{\frac{\mu a_d}{m}\right\}\right), \quad \mu = 1,\dots,m,\quad \K_m(\ba):=\{\by^\mu\}_{\mu=1}^m.
$$
The set $\K_m(\ba)$ is called the {\it Korobov point set}. Further, denote
$$
S(\mathbf k, \ba) \,:=\, P_m\left(e^{i2\pi(\mathbf k,\mathbf x)},\ba\right)
\,=\, m^{-1}\sum_{\mu=1}^{m}e^{i2\pi(\mathbf k,\mathbf y^{\mu})}.
$$
Note that
\be\label{2.1}
P_m (f,\ba) =\sum_{\mathbf k}\hat f(\mathbf k)\, S(\mathbf k,\ba),\quad 
\hat f(\bk) := \int_{[0,1)^d} f(\bx)\, e^{-i2\pi(\bk,\bx)}d\bx,
\ee
where for the sake of simplicity we may assume that $f$ is a
trigonometric polynomial. It is clear that (\ref{2.1}) holds for $f$ with absolutely convergent Fourier series. 

It is easy to see that the following relation holds
\be\label{2.2}
S(\mathbf k,\ba)=
\begin{cases}
1&\quad\text{ for }\quad \mathbf k\in L(m,\ba),\\
0&\quad\text{ for }\quad \mathbf k\notin L(m,\ba),
\end{cases}
\ee
where
$$
L(m,\mathbf a) := \bigl\{ \mathbf k:(\mathbf a,\mathbf k) \equiv 0 \qquad
\pmod{m}\bigr\} .
$$
  For $N\in\N$ define the {\it hyperbolic cross}   by
$$
\Gamma(N,d):= \left\{\bk= (k_1,\dots,k_d)\in\Z^d\colon \prod_{j=1}^d \max(|k_j|,1) \le N\right\}.
$$
Denote 
$$
\Tr(N,d) := \left\{f\, :\, f(\bx)= \sum_{\bk\in \Gamma(N,d)} c_\bk e^{i2\pi (\bk,\bx)}\right\}.
$$
It is easy to see that the condition 
\be\label{ex}
P_m(f,\ba) = \hat f(\mathbf 0), \quad f\in \Tr(N,d),
\ee
is equivalent to the condition
\be\label{exs}
\Gamma(N,d)\cap \bigl(L(m,\ba)\backslash\{\mathbf 0\}\bigr) =
\varnothing.
\ee
\begin{Definition}\label{exact} We say that the Korobov cubature formula $P_m(\cdot,\ba)$ is exact on $\Tr(N,d)$ if condition \eqref{ex} (equivalently, condition \eqref{exs}) is satisfied.
\end{Definition} 

\begin{Theorem}\label{IT.3} Suppose that $P_m(\cdot,\ba)$ is exact on $\Tr(L,d)$ with some $L\in\N$, $L\ge 2$. 
Let $r\in\N$ and $1\le p<\infty$. 
There exist constants $c(d),C(d,p)>0$ such that for any $v\ge c(d)/L$ 
we have 
$$
\tilde D^{r}_p(\K_m(\ba),v) \,\le\, C(d,p)L^{-r} (\log (Lv))^{(d-1)/2}.
$$
\end{Theorem}

\medskip

In  the case $p=\infty$  we prove a weaker upper bound.

\begin{Theorem}\label{IT.4} Let $r\ge 2$. Suppose that $P_m(\cdot,\ba)$ is exact on $\Tr(L,d)$ with some $L\in\N$, $L\ge 2$. 
There exist constants $c(d),C(d)>0$ such that for any $v\ge c(d)/L$ 
we have 
$$
\tilde D^{r}_\infty(\K_m(\ba),v) \,\le\, C(d)L^{-r}(\log (Lv))^{d-1}.
$$
\end{Theorem}

\section{Proofs of Theorems \ref{IT.3} and \ref{IT.4}}
\label{Fib}

The proofs of both theorems go along the same lines. We give a detailed proof of Theorem \ref{IT.3} and point out a change of this proof, which gives Theorem \ref{IT.4}.
For continuous functions of $d$
variables, which are $1$-periodic in each variable, consider the Korobov cubature formula 
 $P_m(\cdot,\ba)$. 
 
For the univariate \emph{test functions} $h^r_u$ we obtain
by the properties of convolution that
$$
\hat h^r_u(y)  = \hat h^{r-1}_u(y)\, \hat h^1_u(y),\qquad y\in\R
$$
which implies for $y\neq 0$
$$
\hat h^r_u(y) = \left(\frac{\sin(\pi yu)}{\pi y}\right)^r.
$$
Therefore,
$$
\left|\hat {\tilde h}^r_u(k)\right| 
\,\le\, \min\left(|u|^r,\frac{1}{|k|^r}\right) 
\,=\, \left(\frac{|u|}{k'}\right)^{r/2}\min\left(|k' u|^{r/2},\frac{1}{|ku|^{r/2}}\right), 
$$
where $k':=\max\{1,|k|\}$. 
(Here, we used for a moment $\hat h$ for the Fourier transform of $h$ on $\R$. 
This should not lead to any confusion.)

\bigskip

It is convenient for us to use the following abbreviated notation for the product
$$
pr(\bu):= pr(\bu,d) := \prod_{j=1}^d u_j.
$$
For $B\subset\Omega_d$ of the form~\eqref{eq:B} and $\bz\in\Omega_d$, we have 
\be\label{eq:shift}
\hat {\tilde h}^r_{B+\bz}(\bk) \,=\, e^{-i2\pi(\bk,\bz)}\, \hat {\tilde h}^r_B(\bk),
\ee
where $\tilde h^r_{B+\bz}(\bx):=\tilde h^r_{B}(\bx-\bz)$, see~\eqref{eq:hB}. 
Therefore, we obtain from the above that
\[
\left|\hat {\tilde h}^r_B(\bk)\right| 
\,\le\, \prod_{j=1}^d \left(\frac{|u_j|}{k_j'}\right)^{r/2}\min\left(|k_j' u_j|^{r/2},\frac{1}{|k_j u_j|^{r/2}}\right).
\]
For $\bs\in \N_0^d$, 
we define
$$
\rho(\bs) := \Big\{\bk \in \Z^d\colon [2^{s_j-1}] \le |k_j| < 2^{s_j}, 
\quad j=1,\dots,d \Big\},
$$
where $[a]$ denotes the integer part of $a$, and obtain, 
for $\bk\in\rho(\bs)$, that  
\be\label{eq:HB}
\left|\hat {\tilde h}^r_B(\bk)\right| 
\,\le\, H_B^r(\bs) 
\,:=\, \left(\frac{pr(\bu)}{2^{\|\bs\|_1}}\right)^{r/2}\,\prod_{j=1}^d \min\left((2^{s_j}u_j)^{r/2},\frac{1}{(2^{s_j}u_j)^{r/2}}\right).
\ee
Later we will need certain sums of these quantities. 
First, consider
$$
\sigma^r_\bu(t) \,:=\, \sum_{\|\bs\|_1=t}\prod_{j=1}^d \min\left((2^{s_j}u_j)^{r/2},\frac{1}{(2^{s_j}u_j)^{r/2}}\right),\quad t\in\N_0.
$$
The following technical lemma is part~(I) from \cite[Lemma 6.1]{VT161}. 

\begin{Lemma}\label{L2.2} Let $r>0$, $t\in \N$ and $\bu\in (0,1/2]^d$ 
be such that $pr(\bu)\ge 2^{-t}$. Then, we have 
\[
\sigma^r_\bu(t) \,\le\, C(d) \frac{\left(\log(2^{t+1}pr(\bu))\right)^{d-1}}{(2^tpr(\bu))^{r/2}}.
\]
\end{Lemma}

\bigskip

\noindent
This lemma and \eqref{eq:HB} imply that
\be\label{eq:HB-bound}
\sum_{\|\bs\|_1=t} H_B^r(\bs)^2 
\,\le\, C_1\, 2^{-2rt}\, \Big(\log\big(2^{t}\, v\big)\Big)^{d-1}, 
\ee
where $v:=\mathrm{vol}(B)=r^d\, pr(\bu)$, 
for all $r\ge1$ and
all $t\in\N_0$ with $v\ge r^d\,2^{-t+1}$
and an absolute constant $C_1<\infty$.

\bigskip

Additionally, we need a result from harmonic analysis -- a corollary of the Littlewood-Paley theorem. Denote
$$
\delta_\bs(f,\bx):= \sum_{\bk\in\rho(\bs)} \hat f(\bk)e^{i2\pi(\bk,\bx)}.
$$
Then it is known that for $p\in [2,\infty)$ one has
\be\label{eq:LP}
\|f\|_p \le C(d,p) \left(\sum_{\bs\in\N_0^d}\|\delta_\bs(f)\|_p^2\right)^{1/2}.
\ee
Note that in the proof of Theorem \ref{IT.4} we use 
the simple triangle inequality
\be\label{eq:triangle}
\|f\|_\infty \le  \sum_{\bs}\|\delta_\bs(f)\|_\infty.
\ee
instead of \eqref{eq:LP}.

\bigskip

Let us define
\be\label{2.3}
E^r_B(\bz):=  \frac{1}{m}\sum_{\mu=1}^{m}  \tilde h^r_B(\by^\mu-\bz)- \int_{[0,1)^d} \tilde h^r_B(\bx)d\bx 
\ee 
such that 
\[
\tilde D^{r}_p(\K_m(\ba),v) \,=\, \sup_{B\subset \Omega_d:vol(B)=v}
	\left\|E^r_B\right\|_p
\]
By formulas \eqref{2.1}, \eqref{2.2} and \eqref{eq:shift} we obtain
\[
E^r_B(\bz) \,=\, \sum_{\bk\neq 0} \hat {\tilde h}^r_B(\bk)\,S(\bk,\ba)\, e^{-i2\pi(\bk,\bz)}
\,=\, \sum_{\bk\in L(m,\ba)\setminus\{0\}} \hat {\tilde h}^r_B(\bk)\, e^{-i2\pi(\bk,\bz)}.
\]
It is apparent from~\eqref{eq:LP} that it remains to bound 
$\|\delta_s(E_B^r)\|_p$. 

If $t\neq 0$ is such that $2^t\le L$ then for
$\bs$ with $\|\bs\|_1=t$ we have $\rho(\bs)\subset \Gamma(L)$. 
Then our assumption \eqref{exs} implies that 
$S(\bk,\ba)=0$ for $\bk\in\rho(\bs)$ and, therefore, 
$\delta_s(E^r_B)=0$. 
Let $t_0\in \N$ be the smallest number satisfying 
$2^{t_0}>L$, 
i.e., $t_0>\log L$.  
Then, from~\eqref{eq:LP} for $p\in[2,\infty)$, we have 
\be\label{eq:EB-bound}
\|E^r_B\|_p \le C(p) \left(\sum_{t=t_0}^\infty \sum_{\|\bs\|_1=t}\|\delta_s(E^r_B)\|_p^2\right)^{1/2}.
\ee
Moreover, \eqref{exs} implies that for $t\ge t_0$ we have
\be\label{eq:number}
\#\big(\rho(\bs)\cap L(m,\ba)\big) \le C_2\, 2^{t-t_0}, \quad \|\bs\|_1=t.
\ee
By Parseval's identity we obtain
\[
\|\delta_s(E^r_B)\|_2 \,=\, \sqrt{\sum_{\bk\in\rho(\bs)\cap L(m,\ba)}|\hat {\tilde h}^r_B(\bk)|^2}
\,\le\, \sqrt{\#\big(\rho(\bs)\cap L(m,\ba)\big)}\cdot H_B^r(\bs)
\]
and, by the triangle inequality, 
\[
\|\delta_s(E_B)\|_\infty \,\le\, \#\big(\rho(\bs)\cap L(m,\ba)\big)\cdot H_B^r(\bs)
\]
Hence, using the inequality 
$$
\|f\|_p \le \|f\|_2^{2/p}\|f\|_\infty^{1-2/p}
$$
for $2\le p\le \infty$, we get 
\[
\|\delta_s(E^r_B)\|_p \,\le\, \Big(\#\big(\rho(\bs)\cap L(m,\ba)\big)\Big)^{1-1/p}\cdot H_B^r(\bs).
\]
Combining this with \eqref{eq:HB-bound}, \eqref{eq:EB-bound} and \eqref{eq:number}, 
we finally obtain for all $v=\mathrm{vol}(B)\ge 2r^d 2^{-t_0}$ and $p\in[2,\infty)$ that
 
\[\begin{split}
\|E^r_B\|_p \,&\le\, C \left(\sum_{t=t_0}^\infty 2^{2(t-t_0)(1-1/p)}  
	\sum_{\|\bs\|_1=t}H_B^r(\bs)^2\right)^{1/2} \\
\,&\le\, C' \left(\sum_{t=t_0}^\infty 2^{2(t-t_0)(1-1/p)}\, 
	2^{-2rt}\, \left(\log\big(2^{t}\, v\big)\right)^{d-1}\right)^{1/2} \\
\,&=\, C'\, 2^{-r t_0} \left(\sum_{t=0}^\infty 2^{2t(1-1/p-r)}\, 
	\left(\log\big(2^{t+t_0}\, v\big)\right)^{d-1}\right)^{1/2} \\
\,&\le\, C''\, 2^{-r t_0}\, \left(\log\big(2^{t_0}\, v\big)\right)^{(d-1)/2} 
	\left(\sum_{t=0}^\infty t\, 2^{2t(1-1/p-r)}\right)^{1/2}.
\end{split}\]
Using $t_0\ge\log L$ and that $\|E^r_B\|_p\le\|E^r_B\|_2$ for $p<2$, 
this implies Theorem~\ref{IT.3}. 
(Here, we used that clearly $r> 1-1/p$ for $p<\infty$.)

As we pointed out above, in the proof of Theorem \ref{IT.4} 
we use inequality \eqref{eq:triangle} instead of \eqref{eq:LP}. 
Moreover we use 
\[
\sum_{\|\bs\|_1=t} H_B^r(\bs) 
\,\le\, C_1\, 2^{-rt}\,\left(\log\big(2^{t}\, v\big)\right)^{d-1}, 
\]
for all $r\ge1$ instead of~\eqref{eq:HB-bound}. 
However, note that we need $r>1$ for the last series in the 
above computation to be finite.
This implies
$$
\|E_B^r\|_\infty \,\le\, C\, L^{-r}\left(\log\left(L\,v\right)\right)^{d-1}. 
$$

\section{Dispersion of the Korobov point sets}
\label{F}

 We remind the definition of 
dispersion. Let $d\ge 2$ and $[0,1)^d$ be the $d$-dimensional unit cube. For $\bx,\by \in [0,1)^d$ with $\bx=(x_1,\dots,x_d)$ and $\by=(y_1,\dots,y_d)$ we write $\bx < \by$ if this inequality holds coordinate-wise. For $\bx<\by$ we write $[\bx,\by)$ for the axis-parallel box $[x_1,y_1)\times\cdots\times[x_d,y_d)$ and define
$$
\cB:= \{[\bx,\by): \bx,\by\in [0,1)^d, \bx<\by\}.
$$
For $n\ge 1$ let $T$ be a set of points in $[0,1)^d$ of cardinality $|T|=n$. The volume of the largest empty (from points of $T$) axis-parallel box, which can be inscribed in $[0,1)^d$, is called the dispersion of $T$:
$$
\text{disp}(T) := \sup_{B\in\cB: B\cap T =\emptyset} vol(B).
$$
An interesting extremal problem is to find (estimate) the minimal dispersion of point sets of fixed cardinality:
$$
\text{disp*}(n,d) := \inf_{T\subset [0,1)^d, |T|=n} \text{disp}(T).
$$
It is known that 
\be\label{F.1}
\text{disp*}(n,d) \le C^*(d)/n.
\ee
Inequality (\ref{F.1}) with $C^*(d)=2^{d-1}\prod_{i=1}^{d-1}p_i$, where $p_i$ denotes the $i$th prime number, was proved in \cite{DJ} (see also \cite{RT}). The authors of \cite{DJ} used the Halton-Hammersly set of $n$ points (see \cite{Mat}). Inequality (\ref{F.1}) with $C^*(d)=2^{7d+1}$ was proved in 
\cite{AHR}. The authors of \cite{AHR}, following G. Larcher, used the $(t,r,d)$-nets (see \cite{NX}, \cite{Mat} for results on $(t,r,d)$-nets and Definition \ref{GD.1} below for the definition). 

\begin{Definition}\label{GD.1} A $(t,r,d)$-net (in base $2$) is a set $T$ of $2^r$ points in 
$[0,1)^d$ such that each dyadic box $[(a_1-1)2^{-s_1},a_12^{-s_1})\times\cdots\times[(a_d-1)2^{-s_d},a_d2^{-s_d})$, $1\le a_j\le 2^{s_j}$, $j=1,\dots,d$, of volume $2^{t-r}$ contains exactly $2^t$ points of $T$.
\end{Definition}

It was demonstrated in \cite{VT161} how good upper bounds on fixed volume discrepancy can be used for proving good upper bounds for dispersion. This fact was one of the motivation for studying the fixed volume discrepancy. 
Theorem \ref{FT.1} below was derived from Theorem \ref{ET.1} (see \cite{VT161}). The upper bound in Theorem \ref{FT.1} combined with 
the trivial lower bound shows that the Fibonacci point set provides optimal rate of decay for the dispersion. 

\begin{Theorem}\label{FT.1} There is an absolute constant $C$ such that for all $n$ we have
\be\label{F.2}
\disp (\cF_n) \le C/b_n.
\ee
\end{Theorem}

The reader can find further recent results on dispersion in \cite{Rud}, \cite{Sos}, \cite{Ull}, and \cite{BH19}. 

We now proceed to new results on dispersion of the Korobov point sets. 

\begin{Theorem}\label{FT.2}   Suppose that $P_m(\cdot,\ba)$ is exact on $\Tr(L,d)$ with some $L\in\N$, $L\ge 2$. 
There exists a constant $C_1(d)>0$ such that   
we have 
$$
\disp(\K_m(\ba)) \,\le\, C_1(d)L^{-1}.
$$
\end{Theorem}
\begin{proof} The proof is based on Theorem \ref{IT.4}. Specify some $r\in \N$, $r\ge 2$. Then, by Theorem \ref{IT.4} we have
$$
\tilde D^{r}_\infty(\K_m(\ba),v) \,\le\, C(d)L^{-r}(\log (Lv))^{d-1}.
$$
 
Let a box $B\in \cB$ be an empty box $B\cap \K_m(\ba) = \varnothing$. Denote $v=vol(B)$.  Note that for $B\in \cB$ we have $\tilde h^r_B(\bx)=h^r_B(\bx)$.  Then, from the Definition \ref{ID.1} of the $\tilde D^{r}_\infty(\K_m(\ba),v)$ it follows that
\be\label{F1}
\left|\int_{\Omega_d}  h^r_B(\bx)d\bx- \frac1m\sum_{\mu=1}^m  h^r_B(\by^\mu)\right|\le \tilde D^{r}_\infty(\K_m(\ba),v).
\ee
Our assumption that $B$ is an empty box implies that
$$
\frac1m\sum_{\mu=1}^m  h^r_B(\by^\mu)=0
$$
and, therefore, by \eqref{F1} we obtain
\be\label{F2}
\left|\int_{\Omega_d}  h^r_B(\bx)d\bx\right|\le \tilde D^{r}_\infty(\K_m(\ba),v).
\ee
If $v\le c(d)L^{-1}$, where $c(d)$ is from Theorem \ref{IT.4}, then Theorem \ref{FT.2} with $C_1(d)\ge c(d)$ follows. Assume $v\ge c(d)L^{-1}$. Then we apply Theorem \ref{IT.4} and obtain from \eqref{F2}
\be\label{F3}
\left|\int_{\Omega_d}  h^r_B(\bx)d\bx\right|\le C(d)L^{-r}(\log (Lv))^{d-1}.
\ee
It follows from the definition of functions $h^r_B(\bx)$ that
\be\label{F4}
\left|\int_{\Omega_d}  h^r_B(\bx)d\bx\right|\ge c_1(d) v^r.
\ee
Inequalities \eqref{F3} and \eqref{F4} imply that $Lv\le C'(d)$. \newline Setting $C_1(d) := \max(c(d),C'(d))$, we complete the proof of Theorem \ref{FT.2}. 
\end{proof}

We now make some comments.

{\bf Fibonacci point sets.} As we already mentioned above we have in the case $d=2$, $m=b_n$, $\ba=(1,b_{n-1})$, that $P_m(f,\ba)=\Phi_n(f)$. Denote in this case $L(n):=L(m,\ba)$.
In other words
$$
L(n) :=\Bigl\{ \mathbf k = (k_1,k_2)\in\Z^2\colon\; k_1 + b_{n-1} k_2\equiv 0
\; \pmod {b_n}\Bigr\}.
$$
The following lemma is well known (see, for instance, \cite{VTbookMA}, p.274).

\begin{Lemma}\label{FL.1} There exists an absolute constant $\gamma > 0$
such that for any $n > 2$ 
for the $2$-dimensional hyperbolic cross 
we have
$$
\Gamma(\gamma b_n,2)\cap \bigl(L(n)\backslash\{\mathbf 0\}\bigr) =
\varnothing.
$$
\end{Lemma}

The combination of Theorem \ref{FT.2} with Lemma \ref{FL.1} implies Theorem \ref{FT.1}.

{\bf Special Korobov point sets.} Let $L\in \N$ be given. Clearly, we are interested in as small $m$ as possible such that there exists a Korobov cubature formula, which is exact on $\Tr(L,d)$. In the case of $d=2$ the Fibonacci cubature formula is an ideal in a certain sense choice. 
There is no known Korobov cubature formulas in case $d\ge 3$, which are as good as  the Fibonacci cubature formula in case $d=2$. We now formulate some known results in this direction. Consider  a special case $\mathbf a = (1,a,a^2,\dots,a^{d-1})$, $a\in\N$. In this case we write in the notation of $\K_m(\ba)$ and $P_m(\cdot,\ba)$ the scalar $a$ instead of the vector $\mathbf a$, namely, $\K_m(a)$ and $P_m(\cdot,a)$. The following Lemma \ref{FL.2} is a simple well known result (see, for instance \cite{VTbookMA}, p.285). 

\begin{Lemma}\label{FL.2} Let $m$ and $L$ be a prime  
and a natural number, respectively, such that
\be\label{F5}
\bigl|\Gamma(L,d)\bigr| < (m-1)/d .
\ee
Then there is a natural number $a\in [1,m)$ such that for all
$\mathbf k\in\Gamma(L,d)$, $\mathbf k\ne\mathbf 0$
\be\label{F6}
k_1 + ak_2 +\dots+a^{d-1}k_d\not\equiv 0\qquad \pmod {m}.
\ee
\end{Lemma}

The combination of Theorem \ref{FT.2} with Lemma \ref{FL.2} implies the following Proposition \ref{FP.1}.

\begin{Proposition}\label{FP.1} There exists a positive constant $C_2(d)$, which depends only on $d$, with the following property. For any $L\in \N$, $L\ge 2$, there exist a prime number $m\le C_2(d)|\Gamma(L,d)|$ and a natural number $a\in[1,m)$ such that 
	$$
	\disp(\K_m(a)) \le C_1(d)L^{-1}.
	$$
\end{Proposition}

\begin{Corollary}

		Let  $C_1(d)$ be the number from Theorem \ref{FT.2} and $p$ be a prime number. There exist a natural number $a \in [1, p)$ such that for any segments of natural numbers $I_1, \dots, I_d \in [1, p]$ satisfying the condition  $$\prod\limits_{j=1}^d |I_j| \ge C_1(d) p^{d-1}(\log{p})^{d-1},$$ there exists a natural number $\mu \in [1, p]$ such that
		
		\begin{equation} \label{system}
		\begin{cases}
		\mu \in I_1 \pmod{p} \\
		\mu a  \in I_2 \pmod{p}\\
		\vdots  \\
		\mu a^{d-1} \in I_d \pmod{p} . 
		\end{cases}
		\end{equation}
\end{Corollary}
	
	\begin{proof}
		Take $$L = \frac{C(d)p}{(\log{p})^{d-1}},$$ where $C(d)$ is small enough to satisfy (\ref{F5}). By Lemma \ref{FL.2} there exists a natural number $a \in [1, p)$, such that $P_p(\cdot, a)$ is exact on $\Tr(L,d)$.
	
		Denote  $ I_j := [x_j, y_j], \  j = 1, \dots, d$. Then for $\tilde{I}_j := [\frac{x_j}{p}, \frac{y_j}{p}]$ we have $$\prod\limits_{j=1}^d |\tilde{I}_j| \ge \frac{C_1(d) (\log{p})^{d-1}}{p}.$$ 
		
		By Theorem \ref{FT.2} the set $\K_p(a)$ intersects the box $\tilde{I}_1 \times \dots \times \tilde{I}_d$ at least at one point. Then, there exists a natural number $\mu \in [1, p]$, such that one has
		\begin{equation}\notag
		\begin{cases}
		\big \{  \frac{\mu }{p}  \} \in \tilde{I}_1, \\
		\big \{  \frac{\mu a}{p}  \} \in \tilde{I}_2,  \\
		\vdots  \\
		\big \{  \frac{\mu a^{d-1}}{p}  \} \in \tilde{I}_d. 
		\end{cases},
		\end{equation}
		which implies (\ref{system}).
		
	\end{proof}

{\bf Acknowledgment.}   The work was supported by the Russian Federation Government Grant No. 14.W03.31.0031.

\end{document}